\documentclass{amsart}

\usepackage{amssymb}
\usepackage{url}
\usepackage{algpseudocode}

\newtheorem{theorem}{Theorem}
\newtheorem{lemma}[theorem]{Lemma}
\newtheorem{proposition}[theorem]{Proposition}
\newtheorem{corollary}[theorem]{Corollary}
\newtheorem{definition}[theorem]{Definition}
\newtheorem{alg}[theorem]{Algorithm}
\theoremstyle{remark}
\newtheorem{remark}[theorem]{Remark}
\newtheorem*{example}{Example}
\numberwithin{theorem}{section}
\numberwithin{equation}{section}

\newcommand{\Z}{\mathbb{Z}}
\newcommand{\Q}{\mathbb{Q}}
\newcommand{\isom}{\mathbb{\cong}}

\newcommand{\calE}{\mathcal{E}}
\newcommand{\calo}{\mathfrak{o}}
\newcommand{\calO}{\mathcal{O}}
\newcommand{\calI}{\mathcal{I}}

\newcommand{\frakp}{\mathfrak{p}}
\newcommand{\frakq}{\mathfrak{q}}
\newcommand{\frakA}{\mathfrak{A}}
\newcommand{\frakP}{\mathfrak{P}}
\newcommand{\Diff}{\mathfrak{D}}
\newcommand{\tensor}{\otimes}
\newcommand{\ol}{\overline}

\DeclareMathOperator*{\bigperp}{\text{\Large $\perp$}}
\DeclareMathOperator{\ord}{ord}
\DeclareMathOperator{\scale}{s}
\DeclareMathOperator{\norm}{n}
\DeclareMathOperator{\Aut}{Aut}
\DeclareMathOperator{\Nr}{Nr}
\DeclareMathOperator{\Tr}{Tr}
\DeclareMathOperator{\End}{End}
\DeclareMathOperator{\GL}{GL}
\DeclareMathOperator{\SL}{SL}
\DeclareMathOperator{\GU}{U}
\DeclareMathOperator{\SU}{SU}
\DeclareMathOperator{\rank}{rank}
\DeclareMathOperator{\gen}{gen}
\DeclareMathOperator{\sgen}{sgen}
\DeclareMathOperator{\Hom}{Hom}

\algnewcommand{\IIf}[1]{\State\algorithmicif\ #1\ \algorithmicthen}
\algnewcommand{\EndIIf}{\unskip\ \algorithmicend\ \algorithmicif}

\begin{document}
\bibliographystyle{abbrv}

\title{Determinant groups of hermitian lattices over local fields}
\author{Markus Kirschmer}
\address{Lehrstuhl D f\"ur Mathematik, RWTH Aachen University, Pontdriesch 14/16, 52062 Aachen, Germany}
\email{markus.kirschmer@math.rwth-aachen.de}
\thanks{The research is supported by the DFG within the framework of the SFB TRR 195.}

\begin{abstract}
We describe the determinants of the automorphism groups of hermitian lattices over local fields.
Using a result of G.~Shimura, this yields an explicit method to compute the special genera in a given genus of hermitian lattices over a number field.
\end{abstract}


\maketitle

\section{Introduction}

An important method to study global fields such as algebraic number fields is to pass to the completions, which are local fields.
In case of a number field, the possible completions are the fields of real or complex numbers as well as the p-adic number fields.
For such fields many problems are much  easier to solve.
The famous local-global principle relates properties of global fields to the respective properties of all its completions.

A classical result known as the Hasse principle shows that quadratic or hermitian spaces over a global field $K$ are isometric if and only if they are isometric over all completions of $K$.
The Hasse principle fails to hold for the analogous arithmetic question, i.e. isometry of lattices over number rings.
This motivates the definition of the genus of a quadratic or hermitian lattice~$L$ as the set of all lattices which are isometric to $L$ locally everywhere.
A genus always decomposes into finitely many isometry classes and it is an important algorithmic task to make this decomposition explicit.

To that end, one considers an equivalence relation which is finer then being in the same genus but coarser than being isometric.
For quadratic lattices, these intermediate equivalence classes are called spinor genera. The analogue for hermitian lattices were dubbed special genera in \cite{Schiemann}.
For lattices in indefinite spaces, strong approximation implies that the spinor or special genera consist of a single isometry class.
In the case of definite spaces, the decomposition of spinor or special genera into isometry classes can be achieved by computing iterated neighbours as defined by Kneser, cf. Remark~\ref{rem:split}.

So it remains to give a description of the spinor or special genera in a given genus.
M.~Kneser \cite{KneserSpinor} answered this question for quadratic lattices.
His result depends on the local spinor norm groups of these lattices, which he computed at all non-dyadic places.
C.~N.~Beli \cite{beli} worked out the local spinor norms at the dyadic places, which makes Kneser's answer explicit.

For hermitian lattices, the question was answered by G.~Shimura in \cite{Shimura}.
For a lattice $L$ of odd rank, his result only depends of the class group of the underlying number field.
If $L$ has even rank however, the answer also depends on the determinants of the local automorphism groups of $L$ and he worked out those groups at all but the ramified dyadic places.

Theorem \ref{main}, which is the main result of this note, gives these local determinant groups at all places.
With Shimura's result and the local determinant groups at hand, Section 4 gives an algorithm to compute representative lattices for the special genera in any given genus of hermitian lattices.

\section{Hermitian spaces}
In this section, we collect some basic results and definitions on lattices in hermitian spaces.
Let $\calo$ be a Dedekind ring with field of fractions $K$ such that the characteristic of $K$ is different from $2$.
Further, let $E = K[X]/(X^2-a)$ be an etale \mbox{$K$-algebra} of dimension $2$ and let $\ol{\phantom{x}} \colon E \to E$ be the $K$-linear automorphism of $E$ with fixed field $K$.
We denote by 
\[ \Nr \colon E \to K, \alpha \mapsto \alpha \ol{\alpha} \quad\mbox{and} \quad  \Tr\colon E \to K, \alpha \mapsto \alpha + \ol{\alpha} \]
the norm and the trace of the $K$-algebra $E$.
Let 
\[ 
\calO = \{ \alpha \in E \mid \Tr(\alpha) \in \calo \mbox{ and } \Nr(\alpha) \in \calo \}
\]
be the integral closure of $\calo$ in $E$ and denote by 
\[ \Diff^{-1}=  \{ \alpha \in E \mid \Tr(\alpha \calO) \subseteq \calo \} \]
the inverse different of $\calO$ over $\calo$.

Let $(V, \Phi)$ be a hermitian space over $E$, i.e. a finitely generated free \mbox{$E$-module}~$V$ and a map
$ \Phi \colon V\times V \to E $
such that for all $v,v', w,w' \in V$ and $\alpha, \beta \in E$ the following conditions hold:
\begin{align*}
  \Phi(\alpha v + \beta v', w) &= \alpha \Phi(v, w) + \beta \Phi(v', w) \:, \\
  \Phi(v, \alpha w + \beta w') &= \ol{\alpha} \Phi(v, w) + \ol{\beta} \Phi(v, w') \:,\\
  \Phi(w,v) &= \ol{\Phi(v,w)} \:.
\end{align*}
The unitary and special unitary groups of $(V, \Phi)$ are
\begin{align*}
\GU(V, \Phi) &= \{\varphi \in \GL(V) \mid \Phi(\varphi(x), \varphi(y)) = \Phi(x,y) \mbox{ for all } x,y \in V \} \mbox{ and }\\
\SU(V, \Phi) &= \GU(V, \Phi) \cap \SL(V)
\end{align*}
Given a basis $(x_1,\dots,x_m)$ of $(V,\Phi)$ over $E$, let $ ( \Phi(x_i, x_j) )_{i,j} \in E^{m \times m}$ be the associated Gram matrix and
\[ \det(V,\Phi) = \det ( \Phi(x_i, x_j) )_{i,j} \in K^* / \Nr(E^*) \]
be the determinant of $(V,\Phi)$. It does not depend on the chosen basis.

An $\calO$-lattice $L$ in $V$ is a finitely generated $\calO$-submodule in $V$ that contains a $K$-basis of $V$.
The fractional $\calO$-ideals
\[ \scale(L)= \{ \Phi(x,y) \mid x,y \in L \} \quad \mbox{and} \quad \norm(L)= \sum_{x \in L} \Phi(x,x) \calO  \]
are called the scale and the norm of $L$ respectively.
For $x,y \in L$ we have 
\[ \Tr(\Phi(x,y)) = \Phi(x+y, x+y) - \Phi(x,x) - \Phi(y,y) \in \norm(L) \]
and thus $\Phi(x,y) \in \Diff^{-1} \cdot \norm(L)$.
This shows that 
\begin{equation}\label{eq:normscale}
 \scale(L) \subseteq \Diff^{-1} \cdot \norm(L) \subseteq \Diff^{-1} \cdot \scale(L) \:. 
\end{equation}

\begin{definition}
Let $L$ be an $\calO$-lattice in $V$.
\begin{enumerate}
\item The lattice $L$ is called maximal if $\norm(L) \subsetneq \norm(M)$ for any $\calO$-lattice $L \subsetneq M$.
\item The dual of $L$ is the $\calO$-lattice $L^\# = \{ x \in V \mid \Phi(x, L) \subseteq \calO \}$.
\item
If $L = \frakA L^\#$ for some fractional ideal $\frakA$  of $\calO$ with $\frakA = \ol{\frakA}$
then $L$ is called $\frakA$-modular (or simply modular).
\end{enumerate}
\end{definition}

\begin{definition}
The unitary group $\GU(V, \Phi)$ acts on the set of all $\calO$-lattices in~$V$.
Two $\calO$-lattices $L$ and $L'$ in $V$ are said to be isometric, denoted by $L \isom L'$, if they lie in the same orbit under $\GU(V, \Phi)$.
The automorphism group $\Aut(L)$ is the stabilizer of $L$ in $\GU(V, \Phi)$.
\end{definition}

\section{Hermitian lattices over local rings}

Let $K, E, \calo, \calO$ and $(V, \Phi)$ be as in Section 2.
We now assume $\calo$ to be the valuation ring of a complete discrete surjective valuation $\ord \colon K \to \Z \cup \{\infty \}$.
The maximal ideal of $\calo$ will be denoted by $\frakp$.
The purpose of this section is to describe the determinants of the automorphisms of $\calO$-lattices in $V$.

If $E \isom K \oplus K$ is split, let $\frakP = \frakp \calO$ otherwise let $\frakP$ be the maximal ideal of $\calO$.
In both cases, $\frakP$ is the largest proper ideal of $\calO$ over $\frakp$ that is invariant under $\ol{\phantom{x}}$.
Hence $\Diff = \frakP^e$ for some integer $e \ge 0$.\\
If $E/K$ is a ramified field extension, we need to distinguish two cases:
\begin{enumerate}
\item $E \isom K[X]/(X^2-p)$ for some prime element $p \in \calo$. Then $\calO$ contains a prime element $\pi$ such that $\ol{\pi} = -\pi$.
In this case $\calO = \calo \oplus \pi \calo$ and $\Diff = 2 \frakP$. In particular, $e = 2 \ord(2) + 1$ is odd.
\item $E \isom K[X]/(X^2-u)$ for some unit $u \in \calo^*$. Section 63A of \cite{OMeara} shows that one may assume that $\ord(u-1) = 2k+1$ for some integer $1 \le k \le \ord(2)$.
Then $\calO = \calo \oplus \alpha \calo$ for some element $\alpha \in \calO$ such that $\alpha^2 = u$.
Let $p \in \calo$ be a prime element. Then $\pi = p^{-k}(1+\alpha)$ is a prime element of $\calO$ and $\Diff = 2 \frakP^{-2k}$. In particular, $e = 2 \ord(2) - 2k$ is even.
\end{enumerate}
Note that the second case can only occur if $E/K$ is wildly ramified, i.e. $E/K$ is ramified and $\ord(2) > 0$.

\begin{lemma}\label{olDiff}
Suppose $E/K$ is ramified. Then $\Diff = \sum_{\gamma \in \calO} (\gamma - \ol{\gamma}) \calO = (\pi - \ol{\pi})\calO$ for any prime element $\pi \in \calO$.
\end{lemma}
\begin{proof}
Suppose first that $\pi$ is the prime element form the discussion just before this lemma.
One checks that $\pi - \ol{\pi}$ generates $\Diff$ in both cases.
The fact that $\calO = \calo \oplus \pi \calo$ shows that any $\gamma \in \calO$ is of the form $\gamma = a + b \pi$ with $a,b \in \calo$.
Then
\[ \gamma - \ol{\gamma} = (a + b \pi) - (a+b\ol{\pi}) = b (\pi - \ol{\pi}) \in \Diff \:. \]
Suppose now $\gamma$ is a prime element of $\calO$. Then $\ord(a) > \ord(b) = 0$ and thus
$(\gamma - \ol{\gamma})\calO = (\pi - \ol{\pi}) \calO = \Diff$. 
\end{proof}

The isometry classes of $\calO$-lattices in $V$ were described by R.~Jacobowitz in \cite{Jacobowitz}, see also \cite{Johnson}.
This classification is not needed for our purposes. We only make us of the following two results.

\begin{proposition}\label{prop1}
Any $\calO$-lattice $L$ in $V$ has a decomposition $L = \bigperp_{i=1}^r L_i$ into modular sublattices $L_i$ of rank $1$ or $2$.
\end{proposition}
\begin{proof}
See for example \cite[Proposition 4.3]{Jacobowitz}.
\end{proof}

If $E/K$ is ramified, let $\pi \in \calO$ be the prime element from above.
For $i \in \Z$, we denote by $H(i)$ a binary hermitian lattice over $\calO$ with Gram~matrix 
\[ \begin{pmatrix} 0 & \pi^i \\ \ol{\pi}^i & 0 \end{pmatrix} \:. \]
For an integer $r \ge 0$ let $H(i)^r$ denote the orthogonal sum of $r$ copies of $H(i)$.

\begin{proposition}\label{prop2}
Suppose $E/K$ is ramified and let $L$ be a $\frakP^i$-modular hermitian $\calO$-lattice of rank $m$.
Then $\norm(L) = \Diff \scale(L)$ if and only if $m$ is even, $i \equiv e \pmod{2}$ and $L \isom H(i)^{m/2}$.
\end{proposition}
\begin{proof}
If $K$ is non-dyadic see \cite[Proposition 8.1]{Jacobowitz}.
Suppose $K$ is dyadic and write $L = \bigperp_j L_j$ with $\frakP^i$-modular lattices of rank at most $2$.
The assumption $\norm(L) = \Diff \scale(L)$ implies that all $L_j$ have rank $2$ and satisfy $\norm(L_j) = \Diff \scale(L_j)$.
So its suffices to discuss the case $m=2$, which is \cite[Proposition~7.1]{Johnson}.
\end{proof}

Given any $\calO$-lattice $L$ in $V$, set
\begin{equation}\label{eq:rho}
 \rho(L)= L + ( \frakP^{-1} L \cap \frakP L^\# ) \:.
\end{equation}
Then $\rho$ defines a map on the set of all $\calO$-lattices in $V$.
It generalizes the maps defined by L. Gerstein in \cite{Gerstein} to hermitian spaces.
They are similar in nature to the \emph{$p$-mappings} introduced by G.~Watson in \cite{pMap}.\\
Since isometries of $L$ also preserve $L^\#$ and commute with sums and intersections, we have $\Aut(L) \subseteq \Aut(\rho(L))$.
Moreover, if $\bigperp_{i=1}^r L_i$ is a decomposition of $L$ into modular sublattices, then
$\rho(L)$ has the decomposition
\[
\bigperp_{i=1}^r L'_i \quad \mbox{where } L'_i =
\begin{cases}
  L_i &\text{if } \frakP \subseteq \scale(L_i), \\
  \frakP^{-1} L_i &\text{otherwise.}
\end{cases}
\]
In particular, $\rho(H(i)) \isom H(i-2)$ for $i \ge 2$ and $\rho(H(i)) = H(i)$ for $i < 2$.

We consider the following subgroups of $\calO^*$.
\begin{align*}
  \calE_0 &= \calE_0(\calO) = \{ u \in \calO^* \mid u \ol{u} = 1 \} \\
  \calE_1 &= \calE_1(\calO) = \{ u \ol{u}^{-1} \mid u \in \calO^* \} \subseteq \calE_0 
\end{align*}
Given any $\calO$-lattice $L$ in $V$ let
\[ \calE(L) = \{ \det(\varphi) \mid \varphi \in \Aut(L) \} \]
be its determinant group. It is a subgroup of $\calE_0$.

\begin{remark}
Suppose $E/K$ is ramified. Hilbert 90 shows that the homomorphism
$ E^* \to \calE_0 / \calE_1, x \mapsto x \ol{x}^{-1} \calE_1 $
is onto.
Its kernel is $K^* \calO^*$ and therefore $[ \calE_0 : \calE_1] = [E^* : K^* \calO^*] = 2$.
\end{remark}

For an anisotropic vector $v \in (V, \Phi)$ and a scalar $\delta \in E$ of norm $1$, we define the corresponding quasi-reflection
\[ \tau_{x,\delta} \colon V \to V,\; y \mapsto y + (\delta-1) \frac{\Phi(y,x)}{\Phi(x,x)} x \:. \]
Note that $\tau_{x,\delta}(x) = \delta x$ and $\tau_{x,\delta}$ is the identity on $\{v \in V \mid \Phi(v,x) = 0\}$. 
Hence $\tau_{x,\delta} \in \GU(V, \Phi)$ and $\det(\tau_{x,\delta}) = \delta$.

Finally, we set
\[
 e' = \max(0, e-1) = 
\begin{cases}
  e-1 & \text{ if $E/K$ is ramified,}\\
  0 & \text{ otherwise.}
\end{cases}
\]

\begin{lemma}\label{E1}
If $\delta \in \calE_0$, then $\delta - 1 \in \frakP^{e'}$.
If $E/K$ is ramified then 
\[ \calE_1 = \{ \delta \in \calE_0 \mid \delta-1 \in \Diff \} \:. \]
\end{lemma}
\begin{proof}
Only the case that $E/K$ is ramified requires proof.
By Hilbert 90, every element of $\calE_0$ is of the form $\beta / \ol{\beta}$ for some $\beta \in E^*$ with $v(\beta) \in \{0, 1\}$.
Suppose $\gamma \in \calO^*$ and let $\pi$ be any prime element of $\calO$.
Lemma \ref{olDiff} shows that
\[
({\gamma} / {\ol{\gamma}} -1)\calO = ( \gamma - \ol{\gamma})\calO \subseteq \Diff \quad \mbox{and} \quad (\pi / \ol{\pi} -1) \calO =  ( \pi   - \ol{\pi}) \frakP^{-1} = \Diff \frakP^{-1}  \:. 
\]
The result follows.
\end{proof}

\begin{corollary}\label{cor0}
Let $L$ be an $\calO$-lattice in $V$. Then $\calE_1 \subseteq \calE(L)$.
If $\scale(L) \frakP^{e'} \subseteq \norm(L)$ then $\calE(L) = \calE_0$.
\end{corollary}
\begin{proof}
Let $x \in L$ be a norm generator, i.e. $\Phi(x,x) \calO = \norm(L)$. 
Suppose first that $\scale(L) \frakP^{e'} \subseteq \norm(L)$ and let $\delta \in \calE_0$.
Lemma \ref{E1} asserts that $(\delta-1) \frac{\Phi(y,x)}{\Phi(x,x)} \in \calO$ for all $y \in L$.
Hence $\tau_{x,\delta} \in \Aut(L)$ and therefore $\delta =  \det(\tau_{x, \delta}) \in \calE(L)$.\\
Suppose now that $\scale(L) \frakP^{e'} \not\subseteq \norm(L)$. Then $E/K$ is ramified.
Let $\delta \in \calE_1$.
Lemma~\ref{E1} and Eq. \eqref{eq:normscale} show that $(\delta-1) \frac{\Phi(y,x)}{\Phi(x,x)} \in \Diff \scale(L) \norm(L)^{-1} \subseteq \calO$ for all $y \in L$.
Hence $\tau_{x,\delta} \in \Aut(L)$ and thus $\delta \in \calE(L)$. 
\end{proof}

We are now ready to give our main result.

\begin{theorem}\label{main}
Let $L$ be an $\calO$-lattice in $V$ and let $m = \rank_E(V)$.
If $E/K$ is ramified, $m$ is even and $L \isom \perp_{i=1}^{m/2} H(s_i)$ with $e \equiv s_i \pmod{2}$ for all $1 \le i \le m/2$, then $\calE(L) = \calE_1$.
In all other cases, $\calE(L) = \calE_0$.
\end{theorem}
\begin{proof}
We fix a decomposition
$L = \bigperp_{i=1}^r L_i$ with modular sublattices $L_i$ of rank $1$ or $2$ as in Proposition \ref{prop1}.
Suppose first that $\calE(L) \ne \calE_0$.
Then $\calE(L_i) \ne \calE_0$ for all $i$.
Corollary \ref{cor0} shows that this is only possible if $E/K$ is ramified and $\Diff \scale(L_i) = \norm(L_i)$.
By Proposition~\ref{prop2} this implies that $L_i \isom H(s_i)$ for some integer $s_i$ with $s_i \equiv e \pmod{2}$.\\
  Conversely assume that $E/K$ is ramified, $m$ is even and $L \isom \perp_{i=1}^{m/2} H(s_i)$ with $e \equiv s_i \pmod{2}$ for all~$i$.
  After rescaling $L$ we may assume that $s_i \ge 0$ for all~$i$.
  Let $j \in \{0,1\}$ such that $j \equiv e \pmod{2}$.
  Repeated application of the map $\rho$ from Eq. \eqref{eq:rho} yields some $\calO$-lattice $M$ in $V$ such that $\Aut(L) \subseteq \Aut(M)$ and $M \isom H(j)^{m/2}$.
  By Corollary \ref{cor0} it suffices to show that $\calE(M) \subseteq \calE_1$.
  There exists some element $\alpha \in E$ such that $\ol{\alpha} = -\alpha$ and $\ord_\frakP(\alpha) = j$.
  Then $H(j) \isom \left( \begin{smallmatrix} 0 & \alpha \\ -\alpha & 0 \end{smallmatrix} \right)$ by Proposition~\ref{prop2}.
  Lemma \ref{olDiff} shows that
  \[ M/\Diff M \times M/\Diff M \to \calO/\Diff,\, (x + \Diff M,y + \Diff M) \mapsto \alpha^{-1} \Phi(x,y) + \Diff \]
  is a well defined symplectic form over $\calO/\Diff \calO$. 
  Since automorphisms of such forms have determinant one, we conclude that $\det(\varphi) \equiv 1 \pmod{\Diff}$ for all $\varphi \in \Aut(M)$.
  Hence  $\calE(M) \subseteq \calE_1$ by Lemma \ref{E1}.
\end{proof}

In \cite[Proposition 4.18]{Shimura} G.~Shimura works out the group $\calE(L)$ for maximal lattices~$L$.
We recover his result from Theorem \ref{main}.

\begin{corollary}[Shimura]
Let $L$ be a maximal $\calO$-lattice in a hermitian space $(V, \Phi)$ over $E$ of rank $m$.
Then $\calE(L) = \calE_1$ if and only if $E/K$ is ramified, $m$ is even and $\det(V,\Phi) = (-1)^{m/2}$.
In all other cases $\calE(L) = \calE_0$.
\end{corollary}
\begin{proof}
  Theorem \ref{main} shows that $\calE(L) = \calE_0$ whenever $E/K$ is unramified or $m$ is odd.
  Suppose now that $E/K$ is ramified and $m$ is even. Let $\norm(L) = \frakP^{2n}$.
  If $\det(V,\Phi) = (-1)^{m/2}$, then $L \isom \bigperp_{i=1}^{m/2}H(2n-e)$, see for example \cite[Propositions~4.7 and 4.8]{Shimura}.
  So in this case, $\calE(L) = \calE_1$.
  Conversely, if $\det(V, \Phi) \ne (-1)^{m/2}$, then $L$ can not be written in the form $L \isom \bigperp_{i=1}^{m/2}H(s_i)$ with integers $s_i$.
  So in this case $\calE(L) = \calE_0$.
\end{proof}

\section{Special genera of hermitian lattices over number fields}

In this section we assume that $K$ and $E$ are both algebraic number fields with ring of integers $\calo$ and $\calO$ respectively.

Let $\frakp$ be a maximal ideal of $\calo$.
The completions of $K$ and $\calo$ at $\frakp$ will be denoted by $K_\frakp$ and $\calo_\frakp$.
More generally, given a vector space $V$ over $K$ and an $\calo$-module $L$, then $V_\frakp = V \tensor_K K_\frakp$ and $L_\frakp = L \tensor_\calo \calo_\frakp$ denote the completions of $V$ and $L$ at $\frakp$.

Let $(V,\Phi)$ be a hermitian space over $E$ of rank at least $2$ and let $L$ be an \mbox{$\calO$-lattice} in $V$.
The space $(V,\Phi)$ is called definite, if $K$ is totally real and there exists some $a \in K^*$ such that $a\Phi(x,x)$ is totally positive for every nonzero vector $x \in V$.
Otherwise $(V,\Phi)$ is called indefinite.\\
By linearity, $\Phi$ extends to $V_\frakp$.
This yields a hermitian space $(V_\frakp, \Phi)$ over the etale $K_\frakp$-algebra $E_\frakp$ which contains the $\calO_\frakp$-lattice $L_\frakp$.

\begin{definition}\label{Def:Genus}
Two lattices $L$ and $L'$ in $V$ are said to be in the same genus if $L_\frakp \isom L'_\frakp$ for every maximal ideal $\frakp$ of $K$.
The lattices are said to be in the same special genus, if there exists an isometry $\sigma \in \GU(V,\Phi)$ such that $L_\frakp = \sigma(\varphi_\frakp(L'_\frakp))$ with $\varphi_\frakp \in \SU(V,\Phi)$ for every maximal ideal $\frakp$ of $\calo$.
The genus and special genus of $L$ will be denoted by $\gen(L)$ and $\sgen(L)$ respectively.
\end{definition}

It is well known that the genus of $L$ is a union of finitely many special genera and each special genus decomposes into finitely many isometry classes.
The special genera in $\gen(L)$ were described by G.~Shimura \cite{Shimura} in terms of the local determinant groups $\calE(L_\frakp)$.
To state his result, some more notation is needed:
\begin{itemize}
\item $\calI$ denotes the group of fractional ideals of $\calO$.
\item $J = \{\frakA \in \calI \mid \frakA \ol{\frakA} = \calO \}$ and $J_0 = \{ \alpha\calO \mid \alpha \in E^* \mbox{ and } \Nr(\alpha) = 1 \} \subseteq J_0$.
\item $C = \calI / \{ \alpha \calO \mid \alpha \in E^* \}$ denotes the class group of $\calO$.
\item $C_0 = \{ [\frakA] \in C \mid \ol{\frakA} = \frakA \} $ is the subgroup of $C$ generated by the image of the class group of $\calo$ in $C$ and
\[ \{ [\frakP] \in C \mid \frakP \in \calI \mbox{ a prime ideal that ramifies over } K \} \:. \]
\item $P(L)$ is the set of all prime ideals of $\calo$ such that $\calE(L_\frakp) \ne \calE_0(\calO_\frakp)$. The ideals in $P(L)$ are necessarily ramified in $E$, cf. Theorem \ref{main}.
\item $\calE(L) =  \prod_{\frakp \in P(L)} \calE_0(\calO_\frakp) / \calE(L_\frakp)$.
\item $H(L) = \{ (e\calO, (e \calE(L_\frakp))_{\frakp \in P(L)}) \in J \times \calE(L) \mid e \in E \mbox{ and } \Nr(e) = 1 \}$.
\item $R(L) = \{ (e \calE(L_\frakp))_{\frakp \in P(L)} \in \calE(L) \mid e \in \calO \mbox{ and } \Nr(e) = 1 \} $.
\item Given a maximal ideal $\frakP$ of $\calO$, define an element $c(\frakP) \in \calE(L)$ by
\[ c(\frakP)_\frakp = \begin{cases}-1 & \text{if } \frakp = \frakP \cap \calO, \\ +1 & \text{otherwise} \end{cases}\] 
for all $\frakp \in P(L)$.
\item Given an $\calO$-lattice $M$ in $V$, let $[L:M]_\calO$ be the fractional ideal of $\calO$ generated by $\{\det(\varphi) \mid \varphi \in \Hom_\calO(L, M) \}$. 
\end{itemize}

\begin{remark}\label{rem:iso}
The group homomorphism $\calI \to J /J_0,\; \frakA \mapsto \frakA \ol{\frakA}^{-1}$ is onto by Hilbert 90.
Hence it induces an isomorphism $C/C_0 \isom J/J_0 $.
\end{remark}

\begin{theorem}[Shimura]\label{Th:Shimura}
Consider the map
\[\Psi \colon \gen(L) \to J \times \calE(L),\, M \mapsto ([L:M]_\calO, ( \det(\sigma_\frakp) \calE(L_\frakp))_{\frakp \in P(L)}) \]
where $\sigma_\frakp \in \GU(V_\frakp, \Phi)$ such that $\sigma_\frakp(L_\frakp) = M_\frakp$ for all $\frakp \in P(L)$.
Then $\Psi$ induces a bijection between the special genera in $\gen(L)$ and $(J \times \calE(L))/H(L)$.
\end{theorem}
\begin{proof}
See \cite[Theorem 5.24 and its proof 5.28]{Shimura}.
\end{proof}

The decomposition given in Theorem \ref{Th:Shimura} can be made explicit using Kneser's concept of neighbours, see also \cite{Schiemann}.

\begin{definition}
Let $\frakP$ be a maximal ideal of $\calO$ and let $\frakp = \frakP \cap \calo$.
Further $L, L'$ be $\calO$-lattices in $V$.
We say that $L'$ is a $\frakP$-neighbour of $L$ if $L_\frakp$ and $L'_\frakp$ are both modular with $\scale(L_\frakp) = \scale(L'_\frakp)$ and 
there exist $\calO$-module isomorphisms
\[ L/(L \cap L') \isom \calO/\frakP \quad \mbox{and} \quad L'/(L \cap L') \isom \calO/\ol{\frakP}\:. \]
\end{definition}

\begin{lemma}\label{lem:cP}
Let $\frakP$ be a maximal ideal of $\calO$ and set $\frakp = \frakP \cap \calo$.
Suppose $L_\frakp$ is modular and if $\frakp$ is ramified in $E$, then $2 \notin \frakp$.
Further, let $\Psi$ denote the bijection from Theorem \ref{Th:Shimura}.
\begin{enumerate}
\item If $(V_\frakp, \Phi)$ is isotropic (which automatically holds if $\rank_E(V) \ge 3$ or $\frakp$ is unramified in $E$), then there exists some $\frakP$-neighbour of $L$.
\item If $L'$ is a $\frakP$-neighbour of $L$ then $L' \in \gen(L)$ and $\Psi(L') = (\frakP \ol{\frakP}^{-1}, c(\frakP)) $.
\end{enumerate}
\end{lemma}
\begin{proof}
Part (1) follows from \cite[Lemma 2.2]{Schiemann} and \cite[Proposition 5.2.4]{Kirschmer}.
Suppose now $L'$ is a $\frakP$-neighbour of $L$.
The definition of $\frakP$-neighbours yields $[L:L']_\calO = \frakP \ol{\frakP}^{-1}$.
Further, $L' \in \gen(L)$ by \cite[Lemma 2.8]{Schiemann} and \cite[Remark 5.2.2]{Kirschmer}.
If $\frakq \in P(L)$ is different from $\frakp$, then $L_\frakq = L'_\frakq$.
Suppose now $\frakp \in P(L)$. Then $E_\frakp / K_\frakp$ is necessarily ramified and there exists some prime element $\pi \in \calO_\frakp$ such that $\ol{\pi} = -\pi$.
Loc. cit. show that there exists a decomposition $L_\frakp = (x \calO_\frakp \oplus y \calO) \perp M$ such that $\Phi(x,x) = \Phi(y,y) = 0$
and $L'_\frakp = (x \frakP^{-1} \oplus y \frakP) \perp M$.
Let $\sigma \in \End_{E_\frakp}(V_\frakp)$ such that $\sigma(x) = -\pi^{-1} x$, $\sigma(y) = \pi y$ and $\sigma(z) = z$ for all $z \in M$.
Then $\sigma \in \GU(V_\frakp, \Phi)$ is an isometry between $L_\frakp$ and $L'_\frakp$ with determinant $-1$.
\end{proof}

The group $J$ is infinite, which makes Theorem \ref{Th:Shimura} difficult to use in practise.
For algorithmic purposes, there is a more convenient description of $(J \times \calE(L))/H(L)$.
To this end, fix some fractional ideals $\frakA_1 = \calO, \frakA_2, \dots, \frakA_r$ of $\calO$ which represent the cosets in $C/C_0$.
For $1 \le i,j \le r$ let $1 \le k(i,j) \le r$  be the unique index such that $ \frakA_i \frakA_j \frakA_{k(i,j)}^{-1} \in C_0$.
Remark \ref{rem:iso} shows that there exists some $\alpha_{i,j} \in E^*$ with $\Nr(\alpha_{i,j}) = 1$ such that
\[ \alpha_{i,j} \frakA_i \ol{\frakA_i}^{-1} \frakA_j \ol{\frakA_j}^{-1} = \frakA_{k(i,j)} \ol{\frakA_{k(i,j)}}^{-1} \:. \]
Then $(\alpha_{i,j})$ defines a 2-cocycle $C/C_0 \times C/C_0 \to \calE(L) / R(L)$.
Let $G(L)$ be the corresponding central extension of $\calE(L) / R(L)$ by $C/C_0$, i.e. $G(L)$ is the cartesian product $C/C_0 \times \calE(L) / R(L)$ equipped with the multiplication 
\[ ( [\frakA_i], x ) * ([\frakA_j], y) = ([\frakA_{k(i,j)}], \alpha_{i,j} xy) \:. \]
The example at the end of this section shows that $G(L)$ does not need to be a split extension of $\calE(L) / R(L)$ by $C/C_0$.

\begin{lemma}[Shimura]\label{GL}
The map
\begin{align*}
\psi \colon (J \times \calE(L)) / H(L) &\to G(L), \; \\
(\alpha \frakA_i \ol{\frakA_i}^{-1}, x) \cdot H(L) &\mapsto ([\frakA_i] \cdot C_0, \alpha^{-1} x \cdot R(L))
\end{align*}
with  $\alpha \in E^*$ an arbitrary element of norm $1$, is an isomorphism of groups. 
In particular, the number of special genera in $\gen(L)$ equals 
\[ \# G(L) = [C:C_0] \cdot [\calE(L) : R(L)] = [J:J_0] \cdot [\calE(L) : R(L)] \:. \]
\end{lemma}
\begin{proof}
The map is well-defined and bijective by \cite[5.28]{Shimura}. It is a group homomorphism by the choice of $G(L)$.
\end{proof}


The group $G(L)$ and the isomorphism $\psi$ from Lemma \ref{GL} yield the following method to decompose a genus into its special genera.

\begin{alg}\mbox{}
\begin{algorithmic}[1]
\Require{An $\calO$-lattice $L$ in a hermitian space $(V, \Phi)$ over $E$.}
\Ensure{A set $S$ of $\calO$-lattices in $V$ representing the special genera in $\gen(L)$.}
\State Compute the groups $\calE(L) / R(L)$ and $C/C_0$.
\IIf{$[\calE(L) : R(L)] = [C:C_0] = 1$}
  \Return $\{L\}$
\EndIIf
\State Using Lemma~\ref{GL}, find prime ideals $\frakP_1,\dots,\frakP_t$ of $\calO$ such that
\begin{enumerate}
\item $\{ g_1,\dots,g_t \}$ generates $G(L)$ where $g_i = \psi((\frakP_i \ol{\frakP_i}^{-1}, c(\frakP_i)) \cdot H(L))$.
\item $(V_{\frakP_i \cap \calo}, \Phi)$ is isotropic and  $L_{\frakP_i \cap \calo}$ is modular.
\item $2 \notin \frakP_i$ if $\frakP_i$ is ramified over $K$.
\end{enumerate}
\For{$1 \le i \le t$}
\State Set $L_{i,-1} = L_{i, 0} = L$.
\State Let $o_i$ be the order of $g_i$ in $ G(L) / \langle g_1,\dots,g_{i-1} \rangle .$
\For{$1 \le j < o_i$}
\State
Let $L_{i,j}$ be a $\frakP_i$-neighbour of $L_{i,j-1}$ different from $L_{i, j-2}$.
\EndFor
\EndFor
\State Set $\frakA = \prod_{i=1}^t \frakP_i^{o_i-1}$.
\State \Return $S = \{ \frakA L + (\ol{\frakA}^{-1} L \cap \bigcap_{i=1}^t L_{i, e_i}) \mid 0 \le e_i < o_i \}$.
\end{algorithmic}
\end{alg}

\begin{proof}[Proof of correctness]
The lattice $M = \frakA L + (\ol{\frakA}^{-1} L \cap \bigcap_{i=1}^t L_{i, e_i})$ in $S$ satisfies
\[ 
M_\frakp = \begin{cases}
  (L_{i, e_i})_\frakp & \text{if }\frakp = \frakP_i \cap \calo,\\
  L_\frakp & \text{otherwise.}
\end{cases}
\]
Thus $G(L) = \{ \psi(\Psi(M)) \mid M \in S \}$ and $\# S = \prod_i o_i = \# G(L)$.
Hence Theorem~\ref{main} and Lemma~\ref{GL} imply that the lattices in $S$ represent each special genus in $\gen(L)$ exactly once.
\end{proof}

The previous algorithms shows how to split a genus into special genera.
To decompose a genus into isometry classes, it remains to describe how to decompose a special genus into isometry classes:

\begin{remark}\label{rem:split}
\begin{enumerate}
\item Suppose $(V, \Phi)$ is definite. Fix a maximal ideal $\frakP$ of $\calO$ such that $(\frakP \ol{\frakP}^{-1}, c(\frakP)) \in H(L)$.
Every isometry class in $\sgen(L)$ has a representative $M$ such that there exists a sequence $L = M_0, M_1, \dots, M_t = M$ where $M_i$ is a $\frakP$-neighbour of $M_{i-1}$, cf. \cite[Corollary 2.7]{Schiemann}.
Conversely, any lattice in such a chain of $\frakP$-neighbours lies in $\sgen(L)$ by the choice of $\frakP$.
So one can decompose $\sgen(L)$ be computing iterated $\frakP$-neighbours. One only needs a method to decide if two definite hermitian $\calO$-lattices are isometric.
The latter can be done by the Plesken-Souvignier algorithm see \cite{PS} and \cite[Section 4.2]{Schiemann}.
\item If $(V, \Phi)$ is indefinite then strong approximation shows that $\sgen(L)$ is a single isometry class, see for example \cite[Corollary~5.1.4]{Kirschmer}.
\end{enumerate}
\end{remark}

\begin{example}\label{ex}
Let $E = \Q(\sqrt{-17})$ and $K = \Q$.
The different of $\calO = \Z[\sqrt{-17}] $ is $\frakP_2^2 \frakP_{17}$ where $\frakP_2$ and $\frakP_{17}$ denote the prime ideals of $\calO$ over $2$ and $17$ respectively.
Fix some prime ideal $\frakP_3$ of $\calO$ over $3$.
The class group  $C$ of $\calO$ is isomorphic to $\Z/4\Z$ and the subgroup $C_0$ has order $2$.\\
Let $L$ be the free hermitian $\calO$-lattice with basis $(x,y)$ 
and associated Gram matrix 
\[
\begin{pmatrix}
102 & \sqrt{-17} \\
-\sqrt{-17} & 0
\end{pmatrix} \:.
\]
Then $L_2 \isom H(0)$ and $L_{17} \isom H(1)$.
Theorem \ref{main} shows that $\calE(L) \isom \Z/2\Z \times \Z/2\Z$. The group $R(L)$ has order $2$ and is diagonally embedded into $\calE(L)$.
Hence the group $G(L)$ has order $[C:C_0] \cdot [\calE(L) : R(L)] = 4$.
Using Lemma \ref{GL}, one checks that $G(L) \isom \Z/4\Z$ is generated by $\psi(\frakP_3 \ol{\frakP}_3^{-1}, 1)$.\\
For $0 \le i \le 3$ let $L_i = \frakP_3^i x \oplus \ol{\frakP_3}^{-i} y$. Then $L_{i}$ is a $\frakP_3$-neighbour of $L_{i-1}$ such that $[L : L_i]_\calO = (\frakP_3 \ol{\frakP}_3)^{-i}$.
Thus the four special genera (or isometry classes) in the genus of $L$ are represented by the lattices $L, L_1, L_2, L_3$.
\end{example}

\subsection*{Acknowledgment}
The author would like to thank S.~Brandhorst for his valuable comments.

\bibliography{det}

\begin{thebibliography}{10}

\bibitem{beli}
C.~N. Beli.
\newblock Integral spinor norm groups over dyadic local fields.
\newblock {\em J. Number Theory}, 102(1):125--182, 2003.

\bibitem{Gerstein}
L.~Gerstein.
\newblock The growth of class numbers of quadratic forms.
\newblock {\em Amer. J. Math.}, 94(1):221--236, 1972.

\bibitem{Jacobowitz}
R.~Jacobowitz.
\newblock Hermitian forms over local fields.
\newblock {\em Amer. J. Math.}, 84:441--465, 1962.

\bibitem{Johnson}
A.~A. Johnson.
\newblock Integral representations of hermitian forms over local fields.
\newblock {\em J. Reine Angew. Math.}, 229:57--80, 1968.

\bibitem{Kirschmer}
M.~Kirschmer.
\newblock {\em Definite quadratic and hermitian forms with small class number}.
\newblock Habilitation, RWTH Aachen University, 2016.

\bibitem{KneserSpinor}
M.~Kneser.
\newblock Klassenzahlen indefiniter quadratischer {F}ormen in drei oder mehr
  {V}er\"anderlichen.
\newblock {\em Arch. Math. (Basel)}, 7:323--332, 1956.

\bibitem{OMeara}
O.~T. O'Meara.
\newblock {\em Introduction to {Q}uadratic {F}orms}.
\newblock Springer, 1973.

\bibitem{PS}
W.~Plesken and B.~Souvignier.
\newblock Computing {I}sometries of {L}attices.
\newblock {\em Journal of Symbolic Computation}, 24:327--334, 1997.

\bibitem{Schiemann}
A.~Schiemann.
\newblock {C}lassification of {H}ermitian {F}orms with the {N}eighbor {M}ethod.
\newblock {\em J. Symbolic Computation}, 26:487--508, 1998.

\bibitem{Shimura}
G.~Shimura.
\newblock Arithmetic of unitary groups.
\newblock {\em Ann. Math.}, 79:269--409, 1964.

\bibitem{pMap}
G.~L. Watson.
\newblock Transformations of a quadratic form which do not increase the
  class-number.
\newblock {\em Proc. London Math. Soc. (3)}, 12:577--587, 1962.

\end{thebibliography}

\end{document}